\documentclass{amsart}

\makeatletter
 \def\LaTeX{\leavevmode L\raise.42ex
   \hbox{\kern-.3em\size{\sf@size}{0pt}\selectfont A}\kern-.15em\TeX}
\makeatother

\newcommand{\BibTeX}{{\rm B\kern-.05em{\sc
i\kern-.025emb}\kern-.08em\TeX}}
\newtheorem{col}{Corollary}[section]
\newtheorem{thm}{Theorem}[section]
\newtheorem{lem}[thm]{Lemma}
\newtheorem{prop}[thm]{Proposition}
\newtheorem{rem}[thm]{Remark}
\theoremstyle{definition}
\newtheorem{defn}{Definition}

\numberwithin{equation}{section}

\begin{document}

\title{$\varphi$-transform  on domains}

\maketitle
\begin{center}

\author{Isaac Z. Pesenson }\footnote{ Department of Mathematics, Temple University,
 Philadelphia,
PA 19122; pesenson@temple.edu. The author was supported in
part by the National Geospatial-Intelligence Agency University
Research Initiative (NURI), grant HM1582-08-1-0019. }

\end{center}

\begin{abstract}
The goal of the present paper is to construct bandlimited  highly localized and nearly tight  frames on 
domains with smooth boundaries in Euclidean spaces. These frames are used do describe corresponding Besov spaces.
\end{abstract}


\maketitle

\section{Introduction}

The goal of the  paper is to describe construction of  bandlimited and highly localized almost tight frames on 
domains with smooth boundaries in Euclidean spaces. These frames are used do describe Besov spaces on domains. Our result is motivated by the well known results of  M. Frazier and  B. Jawerth  about $\varphi$-transform  \cite{FJ1}-\cite{FJW}.

 Let $\Omega\subset \mathbb{R}^{d}$ be a domain with a smooth boundary  $\Gamma$. In the space $L_{2}(\Omega)$ we consider a self-adjoint positive definite  operator $L$ generated by an expression 
 \begin{equation}\label{Op}
 Lf=-\sum_{k,i=1}^{d}\partial_{_{x_{k}}}(a_{k, i}(x)\partial_{x_{i}}f),
\end{equation}
 with zero boundary condition. We will need  the following  family of  cubes
\begin{equation}\label{cube-0}
Q_{k}(\rho)=\left\{x\in \mathbb{R}^{d}: \rho k_{\nu}\leq x_{\nu}\leq \rho  (k_{\nu}+1), \>\nu=1,...,d \right\},
\end{equation}
where $k=(k_{1},...,k_{d})\in \mathbb{Z}^{d}$.

Consider  the sequence 
$
\omega_{j}=2^{j+1},\>\>\>j=0,1,...,
$ 
and pick  a constant $0<\delta<1$.  For  a constant $a_{0}=a_{0}(\Omega, L)>0$  (which appears  in  Theorem \ref{Frame-th}  below)  construct  the sequence 
$
\rho_{j}=a_{0}\delta^{1/d}\omega_{j}^{-1/2},\>\>\>j=0,1,....
$
and introduce  family of  cubes
\begin{equation}\label{cubes}
Q_{k}(\rho_{j}), \>\>\nu=1,...,d,\>\>j=0,1,..., \>\>k=(k_{1},...,k_{d})\in \mathbb{Z}^{d}.
\end{equation}
Let $Q_{i}(\rho_{j}),\>\> \>j=0,1,...,\>\>\>i=1,..., I_{j}, $ be a subcollection of  cubes in (\ref{cubes})  whose intersection with $\Omega$ has positive measure.  If 
\begin{equation}\label{cover-0}
U_{j,i}=Q_{i}(\rho_{j})\cap\Omega, \>\> \>j=0,1,...,\>\>i=1,..., I_{j},
\end{equation}
then   for every fixed $j=0,1,..., $ the collection $\{U_{j,i}\}_{i=1}^{I_{j}},\>\>\>U_{j,i}\subseteq Q_{i}(\rho_{j}),$ will be a disjoint (except for a set of measure zero) cover of $\Omega$.

With every set $U_{j,i},  \>\>\>j=0,1,...,\>\>\>i=1,..., I_{j},$ we will associate a function $\varphi_{j,i}\in L_{2}(\Omega)$ in a way it is described in the following main theorem.

\begin{thm}(Frame Theorem)\label{FRAME-TH}

For the functions $\varphi_{j,i},\>\>j=0,1,...,\>\>i=1,..., I_{j},$ the following holds

\begin{enumerate}

 \item every $ \varphi_{j,i}$ is bandlimited  in the sense that it is a linear combination of eigenfunctions  of $L$ with eigenvalues  in $[2^{2j-2}, 2^{2j+2}]$; 
 
\item every $ \varphi_{j,i}$ is localized in the following sense: 
    for any non-negative integer vector $\alpha=(\alpha_{1},...,\alpha_{d})$  and every sufficiently large $N\in \mathbb{N}$, there exists a $J(\alpha,  N)$ such that for all  $(j,i)$ with $j>J(\alpha,  N)$ and every  $x$ outside of the cube $Q_{ i}(2^{-N-2})$ one has 
   \begin{equation}
     \left|   \frac{\partial^{\alpha}}{\partial x_{1}^{\alpha_{1}}...\partial x_{d}^{\alpha_{d}}} 
\varphi_{j,i}(x)\right|\leq a_{0}(\Omega, L)\delta 2^{-N-\frac{dj}{2}},\>\>j>J(\alpha,  N),\>\>1\leq i\leq I_{j}.
   \end{equation}
 
In addition, there  exists a constant $C$ such that for all pairs $(j,i)$ 
   \begin{equation}
         \|\varphi_{j,i}\|_{L_{2}(\Omega)}\leq C .                              
   \end{equation}

\item $\{\varphi_{j,i}\}$  is a frame in $L_{2}(\Omega)$ with constants $1-\delta$ and $1$, i.e.
 $$
 (1-\delta)\|f\|^{2}_{L_{2}(\Omega)}\leq \sum_{j\geq 0}^{\infty}\>\> \sum_{1\leq i\leq I_{j}} \left|\left<f,\varphi_{j,i}\right>\right|^{2}\leq \|f\|^{2}_{L_{2}(\Omega)},\>\>\>f\in L_{2}(\Omega);
 $$

\end{enumerate}

\end{thm}

In Theorem \ref{mainLemma2} this result is used to describe Besov norm of a function $f\in L_{2}(\Omega)$  in terms of frame coefficients   $\left<f,\varphi_{j,i}\right>$.

We don't discuss any reconstruction method of a function $f\in L_{2}(\Omega)$ from its projections  $\left<f,\varphi_{j,i}\right>$. However, since our frame is "nearly" tight  (at least when $\delta$ is close to zero) in practice one can use the same frame for reconstruction to have $f\approx \sum _{j,i}\left<f, \varphi_{j,i}\right>\varphi_{j,i}$. Another way for reconstruction is to use iterative the so called frame algorithm, which in this case will exhibit geometric convergence with factor $\delta^{n}(2-\delta)^{-n}$, where $n$ is the number of iteration steps. Another possibility for reconstruction  is through interpolation by variational splines exactly as it was done in \cite{Pes00}-\cite{Pes0}.

In section 3 we construct frames in spaces of bandlimited functions $E_{\omega}(L)=span\>\{u_{k}\},\>Lu_{k}=\lambda_{k}u_{k},\>\lambda_{k}\leq \omega,$ in a way that their frame constants are independent on $\omega$ (Theorem \ref{Frame-th}). It is important to note that in Theorem  \ref{Frame-th} which provides  descritization of the norm in a space $E_{\omega}(L)$   a number of "samples" $\Phi_{i}(f)$ is approximately $|\Omega|\omega^{d/2},$  which according to the Weyl's asymptotic formula \cite{Hor}, \cite{T},
$$
dim\>E_{\omega}(L)\sim |\Omega|\omega^{d/2},
$$
 is essentially the dimension of  the space $E_{\omega}(L)$. In this sense Theorem \ref{Frame-th} is optimal.

In section 4 we represent functions in terms of appropriate bandlimited components and apply Theorem \ref{Frame-th}. Localization of frame elements follows from well known properties of spectral projectors for self-adjoint elliptic second-order differential operators on manifolds \cite{Hor}, \cite{T}. In section 5 we introduce Besov spaces as interpolation spaces between $L_{2}(\Omega)$ and domains of powers of $L$ and characterize them in terms of frame coefficients. Our  Theorem \ref{approx} and Theorem 5.3 in this section heavily depend on a powerful results of the general theory of interpolation of linear operators. Direct proofs of even more general results can be found in \cite{Pes09}, \cite{Pes11},  \cite{AI}. A treatment of the full scale of Besov spaces is much more involved and will appear in a separate paper.

The most  important fact for our construction of frames  is that in a space of $\omega$-bandlimited functions $E_{\omega}(L)$ 
the continuous and discrete norms are equivalent. This result in the case of compact and non-compact manifolds of bounded geometry was  discovered and explored in many ways in our papers \cite{Pes98}-\cite{Pes09}. In the classical cases of straight line $\mathbb{R}$ and  circle $\mathbb{S}$ the corresponding results are known as Plancherel-Polya and Marcinkiewicz-Zygmund inequalities. Our generalization of Plancherel-Polya and Marcinkiewicz-Zygmund inequalities  implies  that $\omega$-bandlimited functions  on manifolds of bounded geometry are completely determined by the values of their averages over "small" sets  "uniformly" distributed over $M$ with a spacing comparable to $1/\sqrt{\omega}$ and can be completely reconstructed in a stable way from such sets of values. The last statement is an extension of the Shannon sampling theorem to the case of Riemannian manifolds of bounded geometry.

The present paper is the first systematic development of bandlimited localized frames and their relations to Besov spaces on general domains. Several  approaches to frames on the unit ball in $\mathbb{R}^{d}$ were considered in \cite{PX}, \cite{LMc}, \cite{LPS} but their methods and results are very different from ours. 

 Most of our proofs and results hold for general compact  Riemannian manifolds without boundary and even for non-compact manifolds of bounded geometry. We do not discuss such manifolds in this paper since for, say,  compact closed manifolds  nearly   tight bandlimited and localized frames were already developed in \cite{gm4}. Moreover, in the case of homogeneous compact manifolds bandlimited and localized {\bf tight }frames were constructed in \cite{gpes}. 
  In the following papers a number of frames  was constructed in different function spaces on closed compact  manifolds  and on non-compact manifolds     \cite{CGSM},  \cite{CO1}-\cite{CKP}, \cite{FM1}- \cite{gpes},   \cite{NPW1}-\cite{Pessubm}.     Necessary conditions for sampling and interpolation  in terms of Beurling-Landau densities on compact manifolds were obtained in  \cite{OP}.

  Applications of frames on manifolds to scattering theory, statistic and cosmology  can be found in  \cite{BKMP}, \cite{GM1}, \cite{GM2}, \cite{LMc}, \cite{LPS},  \cite{MP},  \cite{M-all}.  
 
 \textbf{Acknowledgement:} The relevance of frames on manifolds  to cosmology was discussed with Domenico Marinucci during my  visit to the University of Rome Tor Vergata in May 2012. I would like to thank him for inviting me  and for numerous conversations, which stimulated my interest to frames on domains in Euclidean spaces.

  \section{Bounded domains with smooth boundaries and operators}
 
 We consider   bounded domains $\Omega\subset \mathbb{R}^{d}$ with a smooth boundaries  $\Gamma$ which are  smooth $(d-1)$-dimensional oriented manifolds.  Let $\overline{\Omega}=\Omega\cup \Gamma$ and $L_{2}(\Omega)$ be  the  space of functions  square-integrable with respect to Lebesgue  measure $dx=dx_{1}...dx_{d}$ with the norm denoted as $\|\cdot \|$. If $k$ is a natural number the notations $H^{k}(\Omega)$ will be  used for  the Sobolev space of distributions on $\Omega$ with the norm
 $$
 \|f\|_{H^{k
 }(\Omega)}=\left(\|f\|^{2}+\sum _{1\leq |\alpha |\leq k}\|\partial^{\alpha} f\|^{2}\right)^{1/2}
 $$
where $\alpha=(\alpha_{1},...,\alpha_{d})$ and $\partial^{\alpha}$ is a mixed partial derivative 
$$
\left(\frac{\partial}{\partial x_{1}}\right)^{\alpha_{1}}...\left(\frac{\partial}{\partial x_{d}}\right)^{\alpha_{d}}.
$$
Under our assumptions the space $C^{\infty}_{0}(\overline{\Omega})$ of infinitely smooth functions with support in $\overline{\Omega}$  is dense  
in $H^{k}(\Omega)$. Closure in $H^{k}(\Omega)$ of the space   $C_{0}^{\infty}(\Omega)$ of smooth functions with support in $\Omega$ will be denoted as $H_{0}^{k}(\Omega)$.

 Since $\Gamma$ can be treated as a smooth Riemannian manifold one can introduce Sobolev scale of spaces $H^{s}(\Gamma),\>\> s\in \mathbb{R},$ as, for example, the domains of the Laplace-Beltrami operator $\mathcal{L}$ of a Riemannian metric on $\Gamma$.
 
 According to the trace theorem there exists a well defined continuous surjective trace operator 
 $$
 \gamma: H^{s}(\Omega)\rightarrow H^{s-1/2}(\Gamma),\>\>s>1/2,
 $$
 such that for all functions $f$  in $H^{s}(\Omega)$ which are smooth up to the boundary the value $\gamma f$ is simply a restriction of $f$ to $\Gamma$.

 One considers the operator (\ref{Op}) 
with coefficients in $C^{\infty}(\Omega)$ where the matrix $(a_{j,k}(x))$ is real, symmetric  and positive definite on $\overline{\Omega}$.
 The operator $L$ is defined as the Friedrichs extension of $L$, initially defined on $C_{0}^{\infty}(\Omega)$, to the set of all functions $f$ in $H^{2}(\Omega)$ with constrain $\gamma f=0$. The Green formula implies that this operator is self-adjoint.  The domain of its positive square root $L^{1/2}$ is the set of all functions $f$ in $H^{1}(\Omega)$ for which $\gamma f=0$. 
  
 Thus, one obtains a self-adjoint positive definite operator in the Hilbert space $L_{2}(\Omega)$ with a discrete spectrum $0<\lambda_{1}\leq \lambda_{2},...$ which goes to infinity.

\section{Average sampling and bandlimited frames on domains}

Let $Q(\rho),\>Q(2\rho)\subset \mathbb{R}^{d}$ be standard  cubes of diameters $\rho$ and $2\rho$ respectively with centers at zero.  Let $U\subset Q(\rho)$  be a closed set and $d\mu$ be a positive measure on $U$. We will assume that
the total measure of $U$ is finite and not zero, i.e.
$$
0<|U|=\int_{U}d\mu<\infty.
$$
We consider the following distribution on
$C^{\infty}(Q(\rho)),$
\begin{equation}
\Psi(\varphi)=\frac{1}{|U|}\int_{U}\varphi d\mu,\>\>\>|U|=\int_{U}d\mu,\>\>\>\>\varphi \in C_{0}^{\infty}(Q(\rho)).
\end{equation}

Some examples of such  distributions which are of particular
interest to us are the following.

1) Weighted Dirac measures. In this case $U=\{x\},\>\> x\in
Q(\rho),$ measure $d\mu$ is any non-zero number $\mu$ and
$\Psi(f)=\mu \delta_{x}(f)=\mu f(x).$

2) Finite or infinite sequences of Dirac measures $\delta_{j},
x_{j}\in Q(\rho),$ with corresponding weights $\mu_{j}$.
In this case $U=\bigcup_{j}\{x_{j}\}$ and
$$
\Psi(f)=\sum_{j}\mu_{j}\delta_{x_{j}}(f),
$$
where we assume the following
$$
0<|U|=\sum_{j}|\mu_{j}|<\infty,\>\>\>\> U=\bigcup_{j}\{x_{j}\}.
$$

3) $U$ is a smooth submanifold in $Q(\rho)$ of any
codimension  and $d\mu$ is its "surface" measure.

4) $U$ is a measurable subset of $Q(\rho)$ , $d\mu$ is the Lebesgue  measure $dx$, and $|U|\neq 0$.

The following statement is an analog of the Poincar\'e inequality.
 \begin{lem}\label{LPIT}
For any $m>d/2$  there exists a  constant $C=C(d, m)>0$ 
 such that  the following inequality holds true
\begin{equation}\label{LPII}
\|f-\Psi(f)\|^{2}_{L_{2}(U)}\leq
C(d,m)\sum_{1\leq |\alpha|\leq
m}\rho^{2\alpha}\|\partial^{\alpha}f\|^{2}_{L_{2}(Q(2\rho))},
\end{equation}
for all  $f\in H^{m}(\Omega), m>d/2,$ where $\alpha=(
\alpha_{1},...,\alpha_{d})$, and $\partial^{\alpha}f=\partial_{x_{1}}^{\alpha_{1}}...\partial_{x_{d}}^{\alpha_{d}}f$ is a partial
derivative of order $|\alpha|=\alpha_{1}+...+\alpha_{d}$.
\end{lem}

\begin{proof}
For any $f\in
C^{\infty}(\Omega)$ and every
 $x,\>y\in U\subset Q(\rho)$, we have the
following
$$
f(x)=f(y)+\sum_{1\leq|\alpha|\leq m-1} \frac{1}{\alpha
!}\partial^{\alpha}f(y)(x-y) ^{\alpha}+
$$
\begin{equation}\label{TF}
\sum_{|\alpha|=m}\frac{1}{\alpha !}\int_{0}^{\eta}t^{m-1}\partial
^{\alpha}f(y+t\vartheta)\vartheta^{\alpha}dt,
\end{equation}
where $x=(x_{1},...,x_{d}), y=(y_{1},...,y_{d}), \alpha=(
\alpha_{1},...,\alpha_{d}),
(x-y)^{\alpha}=(x_{1}-y_{1})^{\alpha_{1}}...
(x_{d}-y_{d})^{\alpha_{d}}, \eta=\|x-y\|, \vartheta=(x-y)/\eta.$

We integrate over
 $U\subset Q(\rho)$ with respect  to $d\mu( y)$. It gives
$$
f(x)-\Psi(f)=|U|^{-1}\int_{U}\left(\sum_{1\leq|\alpha|\leq
m-1} \frac{1}{\alpha !}\partial^{\alpha}f(y)(x-y)
^{\alpha}\right)d\mu(y)+
$$
$$|U|^{-1}\int_{U}\left(
\sum_{|\alpha|=m}\frac{1}{\alpha!}\int_{0}^{\eta}t^{m-1}\partial
^{\alpha}f(y+t\vartheta)\vartheta^{\alpha}dt\right)d\mu(y).
$$
From here we obtain
$$
\|f-\Psi(f)\|_{L_{2}(U)}\leq
C(m)|U|^{-1}\sum_{1\leq|\alpha|\leq
m-1}\left(\int_{U}\left(
\int_{U}|\partial^{\alpha}f(y)(x-y)
^{\alpha}|d\mu(y)\right)^{2}dx\right)^{1/2}+
$$
\begin{equation}\label{RT}
C(m)|U|^{-1}
\sum_{|\alpha|=m}\left(\int_{U}\left(\int_{U}\left |\int_{0}^{\eta}t^{m-1}\partial
^{\alpha}f(y+t\vartheta)\vartheta^{\alpha}dt\right |d\mu(y)\right)^{2}dx\right)^{1/2}=I+II.
\end{equation}
By Minkowski  we obtain that 
\begin{equation}\label{1term}
I\leq C(d,m)\sum_{1\leq 
|\alpha|\leq m-1}\rho^{|\alpha|}\|\partial^{\alpha}f\|_{L_{2}(Q(2\rho))},\>\>\>m>d/2.
\end{equation}
By the  Schwartz inequality using the  assumption $m>d/2$ one can  obtain the following inequality for $|\alpha|=m$
$$
\left |\int_{0}^{\eta}t^{m-1}\partial
^{\alpha}f(y+t\vartheta)\vartheta^{\alpha}dt\right |\leq
C\eta^{m-d/2}\left(\int_{0}^{\eta}t^{d-1}|\partial^{\alpha}
f(y+t\vartheta)|^{2}dt\right)^{1/2}.
$$
Thus, the Minkowski inequality gives that
$$
II\leq C(m)|U|^{-1}
\sum_{|\alpha|=m}\int_{U}\left(\int_{U}\left |\int_{0}^{\eta}t^{m-1}\partial
^{\alpha}f(y+t\vartheta)\vartheta^{\alpha}dt\right 
|^{2}dx\right)^{1/2}d\mu(y)\leq
$$
$$
C(m)|U|^{-1}\sum_{|\alpha|=m}\left(\int_{U}\left(\int_{U}
\eta^{2m-d}\int_{0}^{\eta}t^{d-1}|\partial^{\alpha}f(y+t\vartheta)|^{2}dt
\right)dx\right)^{1/2}d\mu(y).
$$
We integrate over 
$Q(\rho)$ using the spherical coordinate system
$(\eta,\vartheta).$ Since $\eta\leq \rho$ for $|\alpha|=m$ we obtain
\begin{equation}\label{RT-2}
\int_{0}^{\rho/2}\eta^{d-1}\int_{|\theta|=1}
\left |\int_{0}^{\eta}t^{m-1}\partial
^{\alpha}f(y+t\vartheta)\vartheta^{\alpha}dt\right |^{2} d\vartheta
d\eta\leq
$$
$$C(d,m)\int_{0}^{\rho/2}t^{d-1}\left(\int_{|\theta|=1}\int_{0}^{\rho}
\eta^{2m-d}|\partial^{\alpha}f(y+t\vartheta)|^{2}
\eta^{d-1}d\eta d\vartheta\right)dt\leq
$$
$$
C(d,m)\rho^{2m}\|\partial^{\alpha}
f\|^{2}_{L_{2}(Q(2\rho))}.
\end{equation}
The result follows from (\ref{1term}) and (\ref{RT-2}).

\end{proof}

Let  $\{Q_{k}(\rho)\}$ be a collection of cubes of type (\ref{cube-0}). Thus, two cubes from this family can intersect only over their boundaries.
 Set 
 $
 U_{k}(\rho)=Q_{k}(\rho)\cap \Omega
 $ 
 and let  $\{U_{i}(\rho)\}$ be a subcollection of all $U_{k}(\rho)$ which have positive measure. Obviously,   the collection $\{U_{i}(\rho)\}$ is  a cover of $\Omega$ and $
 diam\>U_{i}(\rho)\leq \sqrt{d}\rho$. Thus,
 \begin{equation}\label{cover}
 U_{i}=U_{i}(\rho)=Q_{i}(\rho)\cap \Omega,\>\>\>\>\> \bigcup_{i}U_{i}=\Omega,\>\>\>\>diam\>U_{i}(\rho)\leq \sqrt{d}\rho.
 \end{equation}
 Next, we introduce a family $\Psi=\{\Psi_{i}\}$ of functionals on $L_{2}(\Omega)$ where every functional  has the form
 \begin{equation}\label{functionals-2}
 \Psi_{i}(f)=\frac{1}{|U_{i}|}\int_{U_{i}}f(x)dx,\>\>\>f\in L_{2}(\Omega),\>\>\>\>|U_{i}|=\int_{U_{i}}dx.
 \end{equation}
 
\begin{lem}\label{GPIT}
For any $m>d/2$ there exist  constants $c=c(\Omega, L, m),\>C=C(\Omega, L, m)$  such that  for 
any given $0<\delta<1$ 
if 
$
\rho< c\delta
$
then the following inequality holds
\begin{equation}\label{GPI}
(1-2\delta/3)\|f\|^{2}_{L_{2}(\Omega)} \leq 
\sum_{i}|U_{i}| |\Psi_{i}(f)|^{2}+C\rho^{2m}\delta^{-1}\|L^{m/2}f\|^{2}_{L_{2}(\Omega)}
\end{equation}
for all $f\in \mathcal{D}(L^{m/2})$.
\end{lem}

\begin{proof}

We will need the  inequality (\ref{ineq-1}) below. 
One has  for all $\alpha>0$
$$
|A|^{2}=|A-B|^{2}+2|A-B||B|+|B|^{2},\>\>\>\>\>
2|A-B||B|\leq\alpha^{-1}|A-B|^{2}+\alpha|B|^{2},
$$
which imply  the inequality 
$$
(1+\alpha)^{-1}|A|^{2}\leq\alpha^{-1}|A-B|^{2}+|B|^{2},\>\>\alpha>0.
$$
If, in addition, $0<\alpha<1$, then    one has 
\begin{equation}\label{ineq-1}
(1-\alpha)|A|^{2}\leq \frac{1}{\alpha}|A-B|^{2}+|B|^{2},\>\>0<\alpha<1.
\end{equation}
Applying  inequality  (\ref{ineq-1}) we obtain
\begin{equation}\label{ineq-2}
(1-\alpha)\|f\|^{2}_{L_{2}(\Omega)}\leq\sum_{i}
(1-\alpha)\|f\|^{2}_{L_{2}(U_{i})}\leq 
$$
$$
\alpha^{-1}\sum_{i}\|f-\Psi_{i}(f)\|^{2} _{L_{2}(U_{i})}+
\sum_{i}|U_{i}||\Psi_{i}(f)|^{2},\>\>\>|U_{i}|=\int_{U_{i}} dx.
\end{equation}
Since $\Omega$ has a smooth boundary, there  exist a linear continuous extension operator (see \cite{LM},  Sec. 8.1)
$$
H^{k}(\Omega)\rightarrow H^{k}(\mathbb{R}^{d}),\>\>\>\>f\rightarrow \widetilde{f}\in  H^{k}(\mathbb{R}^{d}).
$$
Note, (see \cite{Hor},  Sec. 17.5), that the following continuous embedding holds $ \mathcal{D}(L^{k/2})\subset H^{k}(\Omega),\>\>\>\>k\in \mathbb{N}$, holds, where $ \mathcal{D}(L^{k/2})$ is considered with the graph norm. 

Thus, if $f\in \mathcal{D}(L^{m/2})$, then according to Lemma \ref{LPIT} one has for every $i$:
$$
\|f-\Psi_{i}(f)\|^{2}_{L_{2}(U_{i})}=\|\widetilde{f}-\Psi_{i}(\widetilde{f})\|^{2}_{L_{2}(U_{i})}\leq
C(d,m)\sum_{1\leq \alpha\leq
m}\rho^{2\alpha}\|\partial^{\alpha}\widetilde{f}\|^{2}_{L_{2}(Q_{i}(2\rho))}.
$$
Applying  (\ref{ineq-2}) with $\alpha=\delta/3$ and summing over $i$ we obtain the  following 
$$
(1-\delta/3) \|f\|^{2}_{L_{2}(\Omega)}\leq 
\sum_{i}|U_{i}| |\Psi_{i}(f)|^{2}+
\frac{3C(\Omega,m)}{\delta}
\sum_{1\leq j\leq
m}\rho^{2j}\|\widetilde{f}\|^{2}_{H^{j}(\cup_{i}Q_{i}(2\rho))}.
$$
Since there exists a $C(\Omega, m)$ such that for all $1\leq j\leq m$
$$
\|\widetilde{f}\|^{2}_{H^{j}(\cup_{i}Q_{i}(2\rho))}\leq \|\widetilde{f}\|^{2}_{H^{j}(\mathbb{R}^{d})}\leq C(\Omega,m)\|f\|^{2}_{H^{j}(\Omega)}
$$
we obtain
$$
(1-\delta/3) \|f\|^{2}_{L_{2}(\Omega)}\leq \sum_{i}|U_{i}| |\Psi_{i}(f)|^{2}+
\frac{C^{'}(\Omega,m)}{\delta}
\sum_{1\leq j\leq
m}\rho^{2j}\|f\|^{2}_{H^{j}(\Omega)}.
$$
The regularity theorem for the elliptic second-order differential   operator $L$ (see \cite{Hor},  Sec. 17.5)
 \begin{equation}\label{reg}
\|f\|^{2}_{H^{j}(\Omega)}\leq b\left(\|f\|^{2}_{L_{2}(\Omega)}+\|L^{j/2}f\|^{2}_{L_{2}(\Omega)}\right),\>\>f\in \mathcal{D}(L^{m/2}),\>\>b=b(\Omega, L, j),
\end{equation}
and the following interpolation inequality (see \cite{Hor},  Sec. 17.5)
\begin{equation}\label{interpol}
\rho^{2j}\|L^{j/2}f\|^{2}_{L_{2}(\Omega)}\leq 4a^{m-j}\rho^{2m}\|L^{m/2}f\|^{2}_{L_{2}(\Omega)}
+ca^{-j}\|f\|^{2}_{L_{2}(\Omega)}, \>\>\>c=c(\Omega, L, m),
\end{equation}
which holds for any $a,\>\rho>0, \>0\leq j\leq m$, imply that  there exists a constant  $C^{''}=C^{''}(\Omega, L, m)$ such that the next inequality takes place
$$
(1-\delta/3)\|f\|^{2}_{L_{2}(\Omega)}\leq 
\sum_{i}|U_{i}| |\Psi_{i}(f)|^{2}+
$$
$$
C^{''}\left(\rho^{2}\delta^{-1}\|f\|^{2}_{L_{2}(\Omega)}+\rho^{2m}\delta^{-1}\|L^{m/2}f\|^{2}_{L_{2}(\Omega)}+a^{-1}\|f\|^{2}_{L_{2}(\Omega)}\right),
$$
where $ m>d/2.$  By choosing $a=(6C^{''}/\delta)>1$ we obtain, that there exists   a  constant  $C^{'''}=C^{'''}(\Omega, L, m)$ such that for any $0<\delta<1$ and $\>\>\>\rho>0$
$$
(1-\delta/2)\|f\|^{2}_{L_{2}(\Omega)}\leq 
\sum_{i} |U_{i}||\Psi_{i}(f)|^{2}+C^{'''}\left(\rho^{2}\delta^{-1}\|f\|^{2}_{L_{2}(\Omega)}+\rho^{2m}\delta^{-1}\|L^{m/2}f\|^{2}_{L_{2}(\Omega)}\right).
$$
 The last inequality shows, that if for  a given $0<\delta<1$ the value of $\rho $ is choosen such that 
 $$
 \rho<c\delta,\>\>\> \>\>c=\frac{1}{\sqrt{6C^{'''}}},\>\>\>\> C^{'''}=C^{'''}(\Omega, L, m), 
 $$
then we obtain for a $m>d/2$
$$
(1-2\delta/3)\|f\|^{2}_{L_{2}(\Omega)}\leq 
\sum_{i}|U_{i}| |\Psi_{i}(f)|^{2}+C^{'''}\delta^{-1}\rho^{2m}\|L^{m/2}f\|^{2}_{L_{2}(\Omega)}.
$$
Lemma is proved.
\end{proof}

In the space $L_{2}(M)$ we consider the functionals 
\begin{equation}\label{phi-psi}
\Phi_{i}(f)=\sqrt{|U_{i}|}\Psi_{i}(f)=\frac{1}{\sqrt{|U_{i}|}}\int_{U_{i}}f(x)dx,\>\>\>\>|U_{i}|=\int_{U_{i}}dx.
\end{equation}

Since the functionals $
\Phi_{i}(f)$ are continuous on a subspace $E_{\omega}(L)$ they can be identified with certain functions in $E_{\omega}(L)$. The theorem below shows that the corresponding set of functions is  a frame in appropriate subspace of bandlimited functions.
\begin{thm}\label{Frame-th}
There exists a constant  $a_{0}=a_{0}(\Omega,  L)$ such that, if for a given $0<\delta<1$ and an $\omega>0$ one has $
\rho<a_{0}\delta^{1/d}\omega^{-1/2},$ 
and conditions (\ref{cover})  are satisfied, then 
\begin{equation}\label{frame-ineq}
(1-\delta)\|f\|^{2}_{L_{2}(\Omega)}\leq \sum_{i} |\Phi_{i}(f)|^{2}\leq\|f\|^{2}_{L_{2}(\Omega)},\>\>\>0<\delta<1,\>\>f\in E_{\omega}(L),
\end{equation}
where $\Phi_{i}$ are defined in  (\ref{phi-psi}).
\end{thm}

\begin{proof}
By using  the Schwartz inequality 
we obtain the right-hand side of (\ref{frame-ineq})
$$
\sum_{i}|U_{i}| |\Psi_{i}(f)|^{2}=\sum_{i}\frac{|U_{i}|}{|U_{i}|^{2}}\left|\int_{U_{i}}fdx\right|^{2}\leq \sum_{i}\int_{U_{i}}|f|^{2}dx=\|f\|^{2}_{L_{2}(\Omega)},\>\>\>f\in L_{2}(\Omega).
$$
According to the previous lemma, there exist $c=c(\Omega,  L),\>C=C(\Omega, L)$ such that for any $0<\delta<1$ and any $\rho<c\delta$
\begin{equation}\label{ineq-3}
(1-2\delta/3)\|f\|^{2}_{L_{2}(\Omega)}\leq 
\sum_{i}|U_{i}| |\Psi_{i}(f)|^{2}+C\rho^{2d}\delta^{-1}\|L^{d/2}f\|^{2}_{L_{2}(\Omega)}.
\end{equation}
Notice, that if $f\in E_{\omega}(L)$, then the Bernstein inequality holds
\begin{equation}\label{bern}
\|L^{d/2}f\|^{2}_{L_{2}(\Omega)}\leq \omega^{d}\|f\|^{2}_{L_{2}(\Omega)}.
\end{equation}
Inequalities (\ref{ineq-3}) and (\ref{bern})  show that for a certain $a_{0}=a_{0}(\Omega, L)$, if  
$
\rho<a_{0}\delta^{1/d}\omega^{-1/2},
$
 then 
\begin{equation}
(1-\delta)\|f\|^{2}_{L_{2}(\Omega)}\leq \sum_{i}|U_{i}| |\Psi_{i}(f)|^{2},\>\>\>0<\delta<1,\>\>f\in E_{\omega}(L).
\end{equation}
Lemma is proved.
\end{proof}

  \section{Bandlimited localized frames  on domains}
  
  \subsection{Bandlimited frames}\label{cutoff}

 Let $h\in C_{0}^{\infty}(\mathbb{R}_{+})$ be a monotonic function such that $supp\>\>h\subset [0,\>  2], $ and $h(s)=1$ for $s\in [0,\>1], \>0\leq h(s)\leq 1, \>s>0.$ Setting  $Q(s)=h(s)-h(2s)$ implies that $0\leq Q(s)\leq 1, \>\>s\in supp\>Q\subset [2^{-1},\>2].$  Clearly, $supp\>Q(2^{-j}s)\subset [2^{j-1}, 2^{j+1}],\>j\geq 1.$ For the functions
 \begin{equation}\label{partition}
 F(s)=\sqrt{Q(s)},\>\>F_{0}(s)=\sqrt{ h(s)}, \>\>F_{j}(s)=\sqrt{Q(2^{-j}s)},\>\>j\geq 1, \>\>\>
 \end{equation}
 one has 
 \begin{equation}\label{id-0}
 \sum_{j\geq 0}F_{j}^{2}(s)=1.
 \end{equation}
 Operator $L$ has a discrete spectrum $0<\lambda_{1}\leq \lambda_{2}\leq..., $ and a  set of eigenfunctions $\{u_{j}\}$, with $ Lu_{j}=\lambda_{j}u_{j},$
 which forms an orthonormal basis in  $L_{2}(\Omega)$.  The positive square root $\sqrt{L}$ has spectrum $0<\sqrt{\lambda_{1}}\leq \sqrt{\lambda_{2}}\leq..., $ and the same set of eigenfunctions $\{u_{j}\},\>\>\>\sqrt{L}u_{j}=\sqrt{\lambda_{j}}u_{j}$. Clearly,
 $$
 E_{\sigma}\left(\sqrt{L}\right)=\left\{span\>\{u_{j}\}: \sqrt{\lambda_{j}}\leq \sigma\right\}=E_{\sigma^{2}}(L).
 $$
 The spectral theorem allows to consider operators $F_{j}(\sqrt{L})$ which are defined as follows
 $$
 F_{j}\left(\sqrt{L}\right)f(x)=F(2^{-j}\sqrt{L})f(x)=\int_{\Omega}K^{F}_{2^{-j}}(x,y)f(y)dy
 $$
   where  $K^{F}_{2^{-j}}(x,y)$ is a smooth function defined as
 \begin{equation}\label{kernel-1}
K^{F}_{2^{-j}}(x,y)= \sum _{m}F\left(2^{-j}\sqrt{\lambda_{m}}\right)u_{m}(x)\overline{u_{m}}(y).
\end{equation}
From (\ref{id-0}) we obtain 
$
\label{addto1op}
\sum_{j\geq 0}^{\infty} F_{j}^2\left(\sqrt{L}\right) = I,
$
where the sum (of operators) converges strongly on $L_2(\Omega)$. 
By applying both sides of this formula to an $f\in L_{2}(\Omega)$ we have
$$
\sum_{j\geq 0}^{\infty} F_{j}^2(\sqrt{L})f = f
$$
and taking inner product with $f$ gives 
\begin{equation}
\label{norm equality}
\|f\|^2_{L_{2}(\Omega)}=\sum_{j\geq 0}^{\infty} \left<F_{j}^2\left(\sqrt{L}\right)f,f\right>=\sum_{j\geq 0} ^{\infty}\|F_{j}\left(\sqrt{L}\right)f\|^2_{L_{2}(\Omega)} .
\end{equation}
Note, that since function $ F_{j}$ has support in  $
[2^{j-1},\>\>2^{j+1}]$ the function $F_{j}\left(\sqrt{L}\right)f $ is bandlimited to  $
[2^{j-1},\>\>2^{j+1}]$. 
We consider the sequence 
$$
\omega_{j}=2^{j+1},\>\>\>j=0,1,...,
$$ 
and fix a $0<\delta<1$.  For  the constant $a_{0}=a_{0}(\Omega, L)>0$ from  Theorem \ref{Frame-th}   construct  the sequence 
\begin{equation}\label{spacing}
\rho_{j}=a_{0}\delta^{1/d}\omega_{j}^{-1/2}=a_{0}\delta^{1/d}2^{-\frac{j+1}{2}},\>\>\>j=0,1,....
\end{equation}
For any fixed  $j=0,1,...$,  let  $\{U_{j,i}\}_{i=1}^{I_{j}}$ be a cover that  constructed in  (\ref{cover-0}). If  $\{\Psi_{j,i}\}_{i=1}^{I_{j}}$ is the corresponding set of functionals constructed according to   (\ref{functionals-2}), then  the frame inequalities (\ref{frame-ineq}) hold in every space $E_{\omega_{j}}(L)$.
We set 
\begin{equation}\label{Phi-Psi}
\Phi_{j,i}(f)=\sqrt{|U_{j,i}|}\Psi_{j,i}(f).
\end{equation}

\begin{rem}\label{remark} 
In what follows we identify functional $\Psi_{j,i}$ with the function $|U_{j,i}|^{-1}\chi_{j,i}$ where $\chi_{j,i}$ is characteristic function of a set $U_{j,i}$ which is contained in a cube $Q_{j,i}(\rho_{j})$. Then every functional $\Phi_{j,i}$ can be identified with $|U_{j,i}|^{-1/2}\chi_{j,i}$ and in this sense
\begin{equation}\label{Phi}
\int_{\Omega}\Phi_{j,i}dx=\sqrt{|U_{j,i}|}\leq \rho_{j}^{d}=a_{0}\delta 2^{-\frac{d}{2}(j+1)},
\end{equation}
where the constant $a_{0}=a_{0}(\Omega, L)>0$ is from  Theorem \ref{Frame-th}.
\end{rem} 
The double inequality  (\ref{frame-ineq}) imply   the following equivalence   for every $j=0, 1, ..., $ 
\begin{equation}\label{???}
(1-\delta)\left \|\> F_{j}\left(\sqrt{L}\right)f\right\|^{2}_{L_{2}(\Omega)}\leq 
\sum _{ i=1}^{I_{j}}\left|\left< F_{j}\left({\sqrt{ L}}\right)f, \Phi_{j,i}\right>\right|^{2}\leq
 \left\| \>F_{j}\left(\sqrt{ L}\right)f\right\|^{2}_{L_{2}(\Omega)},
\end{equation}
where  $F_{j}\left({\sqrt{L}}\right)f\in E_{\omega_{j}}\left(\sqrt{L}\right)=E_{\omega_{j}^{2}}(L)= E_{2^{2j+2}}(L)$. Summing over $j$ and applying  (\ref{norm equality})  gives for any $f\in L_{2}(\Omega)$ the following inequalities
\begin{equation}\label{component:equviv}
(1-\delta)\|f\|^{2}_{L_{2}(\Omega)}\leq\sum_{j\geq 0}^{\infty}\>\>\sum_{i=1}^{I_{j}} \left|\left<F_{j}\left({\sqrt{L}}\right)f, \Phi_{j,i}\right>\right|^{2}\leq \|f\|^{2}_{L_{2}(\Omega)},\>\>\>\>f\in L_{2}(\Omega).
\end{equation}
Since operator $F_{j}\left({\sqrt{ L}}\right)$ is self-adjoint 
we obtain  that for
\begin{equation}\label{FRAME}
\varphi_{j,i}=F_{j}\left({\sqrt{L}}\right)\Phi_{j,i}\in E_{\omega_{j}}\left(\sqrt{L}\right)=E_{2^{2j+2}}(L),
\end{equation}
the following double inequality holds for every $f\in L_{2}(\Omega)$ 
\begin{equation}\label{frame-ineq-10}
(1-\delta)\|f\|^{2}_{L_{2}(\Omega)}\leq\sum_{j\geq 0}\>\>\sum _{i=1}^{I_{j}}\left|\left<f, \varphi_{j,i}\right>\right|^{2}\leq \|f\|^{2}_{L_{2}(\Omega)},\>\>\>\>f\in L_{2}(\Omega), 
\end{equation}
which shows that $\left\{\varphi_{j,i}\right\}$ is a frame in $L_{2}(\Omega)$ .
Let us summarize results of this subsection.
\begin{thm}\label{BD-FRAME}
Consider the sequence 
$
\omega_{j}=2^{j+1},\>\>\>j=0,1,...,
$ 
and pick  a constant $0<\delta<1$.  For  the constant $a_{0}=a_{0}(\Omega, L)>0$ from  Theorem \ref{Frame-th}   construct  the sequence 
$
\rho_{j}=a_{0}\delta^{1/d}\omega_{j}^{-1/2},\>\>\>j=0,1,....
$
and let  the  collection $\{U_{j,i}\}_{i=1}^{I_{j}},\>\>j=0,1,..., \>\>U_{j,i}\subseteq Q_{i}(\rho_{j}),$ be a disjoint cover of $\Omega$ constructed in (\ref{cover-0}).

If every $\varphi_{j,i},\>\>j=0,1,..., \>\>1\leq i\leq I_{j}$ is given by (\ref{FRAME}) then $\varphi_{j,i}\in E_{\omega_{j}}\left(\sqrt{L}\right)=E_{2^{2j+2}}(L),\>\>j=0,1,..., \>\>1\leq i\leq I_{j}$,  is a bandlimited frame in $L_{2}(\Omega)$ with constants $1-\delta$ and $1$.
\end{thm}

\subsection{ Localization of frame functions} 
The last statement is not very useful unless frame functions $\left\{\varphi_{j,i}\right\},\>\>j=0,1,..., \>\>1\leq i\leq I_{j}$,  exhibit certain localization. It is the goal of this subsection to demonstrate that this functions have very strong localization and for  large values of $j=0,1,...,$ they are  essentially concentrated in some neighborhoods of corresponding sets $U_{j,i}$.

Assume that   $g\in C_{0}^{\infty}(\mathbb{R}_{+})$ is a monotonic function such that $supp\>g\subset [0,\>  2], $ and $g(s)=1$ for $s\in [0,\>1], \>0\leq g(s)\leq 1, \>s>0.$  For $t>0$ the function $g_{t}(s)=g(ts)$ has support in $[0, \>2t^{-1}]$ and $g_{t}(s)=1$ for $s\in [0,\>t^{-1}]$. 

We consider  a self-adjoint bounded operator $g\left(t\sqrt{L}\right)$  in $L_{2}(\Omega)$ defined as
\begin{equation}\label{operator}
\left[g\left(t\sqrt{L}\right)f\right](x)=\int_{\Omega}K^{g}_{t}(x,y)f(y)dy,\>\>f\in L_{2}(\Omega),
\end{equation}
 where 
 \begin{equation}\label{kernel-2}
 K^{g}_{t}(x,y)=\sum _{\lambda_{m}}g(t\sqrt{\lambda_{m}})u_{m}(x)\overline{u_{m}}(y).
 \end{equation}
According to the spectral theorem norms of all operators $g\left(t\sqrt{L}\right),\>\>t\geq 0,$ are uniformly bounded.

Note, that  if $\theta\in C_{0}^{\infty}(\Omega)$ has Fourier series $\sum_{j}c_{j}(\theta)u_{j}$ then
$$
\left<\sum_{\lambda_{m}}u_{m}(x)\overline{u_{m}}(y), \theta(y)\right>=\int_{\Omega}\sum_{\lambda_{m}}u_{m}(x)\overline{u_{m}}(y)\theta(y)dy=\sum_{j}c_{j}(\theta)u_{j}(x).
$$
It shows that for every fixed $x\in \Omega$
$$
\delta_{x}(y)=\sum_{\lambda_{m}}u_{m}(x)\overline{u_{m}}(y),
$$
where $\delta_{x}$ is the Dirac measure concentrated at $x\in \Omega$ and convergence of the series is understood in the sense of distributions.
\begin{lem} \label{Lemma}
If  $g$ is the same as above then  for $\theta\in C_{0}^{\infty}(\Omega)$ the function 
$
g\left(t\sqrt{L}\right)\theta
$
goes to $\theta $ in the topology of  $C_{0}^{\infty}(\Omega)$ when $t$ goes to zero. It means that 
$
g\left(t\sqrt{L}\right)\theta
$
goes to $\theta$  uniformly with all derivatives on compact subsets of $\Omega$.
\end{lem}
 \begin{proof} For every   derivative $\partial_{x}^{\alpha}= \frac{\partial^{\alpha}}{\partial x_{1}^{\alpha_{1}}...\partial x_{d}^{\alpha_{d}}},\>\>\>\alpha=(\alpha_{1},...,\alpha_{d}),$ one has
$$
\left|\partial_{x}^{\alpha}g\left(t\sqrt{L}\right)\theta(x)-\partial_{x}^{\alpha}\theta(x)\right|=\left|\sum_{\lambda_{m}}\left[g\left(t\sqrt{\lambda_{m}}\right)-1\right]c_{m}(\theta)\partial_{x}^{\alpha}u_{m}(x)\right|=
$$
$$
\left|\sum_{m,\>\>\lambda_{m}>t^{-1}}c_{m}(\theta)\partial_{x}^{\alpha}u_{m}(x)\right|.
$$
 For eigenvalues the following relation holds $\lambda_{m}\sim m^{2/d}$ \cite{Hor}, \cite{T}. Then an application of the Sobolev embedding theorem gives that for every natural $r$ there exist constants $C_{r}, \>\gamma_{r}$ such that for all natural $m$ one has the inequalities $\|u_{m}\|_{C^{r}(\Omega)}\leq C_{r}(m+1)^{\gamma_{r}}$. Thus, for sufficiently large $l>\alpha+1$ we have
$$
\left|\partial_{x}^{\alpha}g\left(t\sqrt{L}\right)\theta(x)-\partial_{x}^{\alpha}\theta(x)\right|=
\left|\sum_{m,\>\>\lambda_{m}>t^{-1}}(m+1)^{l}c_{m}(\theta)\frac{\partial_{x}^{\alpha}u_{m}(x)}{(m+1)^{l}}\right|\leq
$$
$$
C_{\alpha}\left(\sum_{m,\>\>\lambda_{m}>t^{-1}}(m+1)^{2l}c^{2}_{m}(\theta)\right)^{1/2} \left( \sum_{m,\>\>\lambda_{m}>t^{-1}}(m+1)^{2(\gamma_{\alpha}-l)}\right)^{1/2}, \>\>\theta\in C_{0}^{\infty}(\Omega).
$$
Since $C_{0}^{\infty}(\Omega)$ is a subset of domain of any power of $L$ both  sums on the right goo  to zero when $t\rightarrow 0.$  Lemma is proved. 
\end{proof}

 This Lemma is actually a  main ingredient  of the proofs of a general Propositions 3.5,  3.6 in \cite{T},  Ch. XII. 
Below we reformulate a part of these Propositions  in a form which is more suitable for our purposes.

 In what follows the following notation will be used: 
 $
 \Delta=\{(x,x): x\in \Omega\}.
 $   \begin{prop}\label{prop}
   If $g$ is the same as above  then for any non-negative integer vectors $\alpha=(\alpha_{1},...,\alpha_{d})$ and $\beta=(\beta_{1},...,\beta_{d})$ the derivative 
   $$
   \frac{\partial^{\alpha}}{\partial x_{1}^{\alpha_{1}}...\partial x_{d}^{\alpha_{d}}}    \frac{\partial^{\beta}}{\partial y_{1}^{\beta_{1}}...\partial y_{d}^{\beta_{d}}}K_{t}^{g}(x,y)=\partial_{x}^{\alpha}\partial_{y}^{\beta}K_{t}^{g}(x,y)
   $$ goes to zero uniformly on compact subsets of $(\Omega\times \Omega)\setminus \Delta$ when $t$ goes to zero. 
   
   At the same time
\begin{equation}
K_{t}^{g}(x,x)\sim c\>t^{-d},\>\>\>t\rightarrow 0.
\end{equation}

    \end{prop}
   \begin{proof}
   We sketch the proof only in the case $\alpha=|\beta|=0$ and for the general case refer to the above mentioned reference.  
   
   One clearly has  that $K_{t}^{g}(x,y)=g\left (t\sqrt{L}\right)\delta_{y}$. Pick $\psi_{1}, \psi_{2} \in C_{0}^{\infty}(\Omega)$ which have disjoint supports. The function 
   $$
   \psi_{1}(x)K^{g}_{t}(x,y)\psi_{2}(y)=R_{t}^{g}(x,y)
   $$
   is the kernel of the operator $\psi_{1}(x)g\left(t\sqrt{L}\right)(\psi_{2} f)=R^{g}_{t} f$.  Next, we note that by using duality one can define operators $R^{g}_{t}$ on the set of distributions on $\Omega$. If $v$ is a distribution on $\Omega$ then the set $\left\{R_{t}^{g}v\right\} $ with $0<t<1$ is bounded in $C^{\infty}(\Omega)$ and along with the previous Lemma it implies that $R_{t}^{g}v$ goes to zero in $C^{\infty}(\Omega)$ uniformly on compact sets of distributions $v$.  In particular it is true for  the set of distributions $\left\{\delta_{y}\right\}_{y\in \Omega}$.  In other words the kernel $R_{t}^{g}(x,y)$ goes to zero uniformly when $t$ goes to zero.
   It proves the proposition in the case $\alpha=|\beta|=0$.
    \end{proof}
   
   \begin{col}\label{col}
    Given  two non-negative integer vectors $\alpha=(\alpha_{1},...,\alpha_{d})$ and $\beta=(\beta_{1},...,\beta_{d})$  and   two sufficiently small positive numbers $\epsilon_{1}, \epsilon_{2}$  one can find a positive  $t_{0}(\alpha, \beta, \epsilon_{1}, \epsilon_{2})$ such that for any $x\in\Omega $ the following inequality holds 
   \begin{equation}\label{K-111}
   \left|\partial_{x}^{\alpha}\partial_{y}^{\beta}K^{g}_{t}(x,y)\right|<\epsilon_{2}
   \end{equation}
   for all  $y\in \Omega\setminus B(x, \epsilon_{1})$ as long as $0<t<t_{0}(\alpha, \beta, \epsilon_{1}, \epsilon_{2})$. Equivalently, for any $y\in\Omega $ the same inequality (\ref{K-111}) holds for all    $x\in \Omega\setminus B(y, \epsilon_{1})$ as long as $0<t<t_{0}(\alpha, \beta, \epsilon_{1}, \epsilon_{2})$. Here $B(x, \epsilon_{1}),\>\>B(y, \epsilon_{1})$ are balls with centers $x$ and $y$ and radius  $\epsilon_{1}$. 
   
   \end{col}
   
      We now return to our situation with kernels $K^{F}_{2^{-j}}(x,y)$ defined in (\ref{kernel-1}).
      
   \begin{thm}\label{LOCALIZATION}
    For any non-negative integer vectors $\alpha=(\alpha_{1},...,\alpha_{d})$  and every sufficiently large $N\in \mathbb{N}$, there exists a $J(\alpha, N)$ such that for all  $(j,i)$ with $j>J(\alpha, N)$ and every  $x$ outside of the cube $Q_{i}(2^{-N-2})$ one has 
   \begin{equation}
      \left| \partial_{x}^{\alpha}\varphi_{j,i}(x)\right|\leq a_{0}\delta2^{-N-dj/2},
   \end{equation}
where the constant $a_{0}=a_{0}(\Omega, L)>0$ is the same as in   Theorem \ref{Frame-th}.

  Also, there  exists a constant $C$ such that for all pairs $(j,i)$ 
   \begin{equation}\label{norm}
         \|\varphi_{j,i}\|_{L_{2}(\Omega)}\leq C .                              
   \end{equation} 

\end{thm}
\begin{proof}First,  we note that for the proofs of Lemma \ref{Lemma} and Proposition \ref{prop} it was important that function $g\in C_{0}^{\infty}(\mathbb{R}_{+})$ takes value one in a neighborhood of zero.  At the same time  the support of the function $F$ is in $[2^{-1}, \>2]$ (see (\ref{partition})). However, it is clear that $F$ is a difference of two functions  $g_{1}, \>g_{2}\in C_{0}^{\infty}(\mathbb{R}_{+})$ which take value $1$ in some neighborhoods of zero. Since $K^{F}_{t}(x,y)=K^{g_{1}}_{t}(x,y)-K^{g_{2}}_{t}(x,y)$ the previous Corollary  \ref{col} and in particular inequality (\ref{K-111}) hold true for the kernel 
   \begin{equation}\label{}
K^{F}_{t}(x,y)= \sum _{m}F\left(t\sqrt{\lambda_{m}}\right)u_{m}(x)\overline{u_{m}}(y).
\end{equation}
 Pick a large  $N\in \mathbb{N}$ and set $\epsilon_{1}=\epsilon_{2}=2^{-N-2}$.  With this choice of $\epsilon_{1},\>\epsilon_{2}$ let $J(\alpha, N)$  be a smallest natural number  for which $2^{-J(\alpha,  N)}<t_{0}(\alpha,  \epsilon_{1}, \epsilon_{2})$, where $t_{0}(\alpha,   \epsilon_{1}, \epsilon_{2})$ is from the previous Corollary. 
 
 Below we are using a family of cubes  $Q_{ i}(\rho_{j})$ and a family of sets $U_{j,i}(\rho_{j})=Q_{i}(\rho_{j})\cap \Omega\subseteq Q_{i}(\rho_{j})$ which were defined in (\ref{cube-0})-(\ref{cover-0}).

 Since support of $\Phi_{j,i}=|U_{j,i}|^{-1/2}\chi_{j,i}$ is the set $U_{j,i}\subseteq Q_{i}(\rho_{j})$ 
 we obtain by using (\ref{Phi})  that for every $j>J(\alpha,  N)$ and every $x$ outside of the cube $Q_{i}(\rho_{j})$ with $\rho_{j}=a_{0}\delta^{1/d}2^{-\frac{j+1}{2}},\>\>\>j=0,1,....$ the next inequality holds
 \begin{equation}
\left| \partial_{x}^{\alpha}\varphi_{j,i}(x)\right|=\left|\partial_{x}^{\alpha}F_{j}(L)\Phi_{j,i}(x)\right| =\left|\int_{U_{j,i}}\partial_{x}^{\alpha}K^{F}_{2^{-j}}(x,y)\Phi_{j,i}(y)dy\right|\leq
 $$
 $$
 a_{0}\delta 2^{\frac{-d(j+1)}{2}}\sup_{y\in U_{j,i}}\left|\partial_{x}^{\alpha}K^{F}_{2^{-j}}(x,y)\right| \leq a_{0}\delta 2^{-N-2}2^{\frac{-d(j+1)}{2}}\leq a_{0}\delta 2^{-N-\frac{dj}{2}}.
  \end{equation}
The inequality (\ref{norm}) follows from (\ref{FRAME}), the fact that  $F_{j}\left(\sqrt{L}\right)$ is a bounded operator in $L_{2}(\Omega)$ and the formula $\Phi_{j,i}=|U_{j,i}|^{-1/2}\chi_{j,i}$. Theorem  \ref{LOCALIZATION} is proved. 
\end{proof}
    
Theorems \ref{BD-FRAME} and \ref{LOCALIZATION} imply Theorem \ref{FRAME-TH}.    
    
\section{Besov spaces}\label{Besov}

We are going to remind a few basic facts from the theory of interpolation and  approximation spaces spaces \cite{BL}, \cite{BB}, \cite{KPS}.

Let $E$ be a linear space. A quasi-norm $\|\cdot\|_{E}$ on $E$ is
a real-valued function on $E$ such that for any $f,f_{1}, f_{2}\in
E$ the following holds true

\begin{enumerate}
\item $\|f\|_{E}\geq 0;$

\item $\|f\|_{E}=0  \Longleftrightarrow   f=0;$

\item $\|-f\|_{E}=\|f\|_{E};$

\item $\|f_{1}+f_{2}\|_{E}\leq C_{E}(\|f_{1}\|_{E}+\|f_{2}\|_{E}),
C_{E}>1.$

\end{enumerate}

We say that two quasi-normed linear spaces $E$ and $F$ form a
pair, if they are linear subspaces of a linear space $\mathcal{A}$
and the conditions $\|f_{k}-g\|_{E}\rightarrow 0,$ and
$\|f_{k}-h\|_{F}\rightarrow 0, f_{k}, g, h \in \mathcal{A}, $
imply equality $g=h$. For a such pair $E,F$ one can construct a
new quasi-normed linear space $E\bigcap F$ with quasi-norm
$$
\|f\|_{E\bigcap F}=\max\left(\|f\|_{E},\|f\|_{F}\right)
$$
and another one $E+F$ with the quasi-norm
$$
\|f\|_{E+ F}=\inf_{f=f_{0}+f_{1},f_{0}\in E, f_{1}\in
F}\left(\|f_{0}\|_{E}+\|f_{1}\|_{F}\right).
$$

All quasi-normed spaces $H$ for which $E\bigcap F\subset H \subset
E+F$ are called intermediate between $E$ and $F$. A group
homomorphism $T: E\rightarrow F$ is called bounded if
$$
\|T\|=\sup_{f\in E,f\neq 0}\|Tf\|_{F}/\|f\|_{E}<\infty.
$$
One says that an intermediate quasi-normed linear space $H$
interpolates between $E$ and $F$ if every bounded homomorphism $T:
E+F\rightarrow E+F$ which is a bounded homomorphism of $E$ into
$E$ and a bounded homomorphism of $F$ into $F$ is also a bounded
homomorphism of $H$ into $H$.

On $E+F$ one considers the so-called Peetere's $K$-functional
\begin{equation}
K(f, t)=K(f, t,E, F)=\inf_{f=f_{0}+f_{1},f_{0}\in E, f_{1}\in
F}\left(\|f_{0}\|_{E}+t\|f_{1}\|_{F}\right).\label{K}
\end{equation}
The quasi-normed linear space $(E,F)^{K}_{\theta,q}, 0<\theta<1,
0<q\leq \infty,$ or $0\leq\theta\leq 1,  q= \infty,$ is introduced
as a set of elements $f$ in $E+F$ for which
\begin{equation}
\|f\|_{\theta,q}=\left(\int_{0}^{\infty}
\left(t^{-\theta}K(f,t)\right)^{q}\frac{dt}{t}\right)^{1/q}.\label{Knorm}
\end{equation}

It turns out that $(E,F)^{K}_{\theta,q}, 0<\theta<1, 0\leq q\leq
\infty,$ or $0\leq\theta\leq 1,  q= \infty,$ with the quasi-norm
(\ref{Knorm})  interpolates between $E$ and $F$. 

Let us introduce another functional on $E+F$, where $E$ and $F$
form a pair of quasi-normed linear spaces
$$
\mathcal{E}(f, t)=\mathcal{E}(f, t,E, F)=\inf_{g\in F,
\|g\|_{F}\leq t}\|f-g\|_{E}.
$$

\begin{defn}The approximation space $\mathcal{E}_{\alpha,q}(E, F),
0<\alpha<\infty, 0<q\leq \infty $ is a quasi-normed linear spaces
of all $f\in E+F$ with the following quasi-norm
\begin{equation}
\left(\int_{0}^{\infty}\left(t^{\alpha}\mathcal{E}(f,
t)\right)^{q}\frac{dt}{t}\right)^{1/q}.
\end{equation}
\end{defn}

\begin{thm}\label{equivalence}
 Suppose that $\mathcal{T}\subset F\subset E$ are quasi-normed
linear spaces and $E$ and $F$ are complete.

 If there
exist $C>0$ and $\beta >0$ such that for any $f\in F$ the
following Jackson-type inequality is verified
\begin{equation}
t^{\beta}\mathcal{E}(t,f,\mathcal{T},E)\leq C\|f\|_{F},\label{dir}
t>0,
\end{equation}
 then the following embedding holds true
\begin{equation}
(E,F)^{K}_{\theta,q}\subset \mathcal{E}_{\theta\beta,q}(E,
\mathcal{T}), \>0<\theta<1, \>0<q\leq \infty.
\end{equation}

If there exist $C>0$ and $\beta>0$ such that for any $f\in \mathcal{T}$ the following Bernstein-type inequality holds
\begin{equation}
\|f\|_{F}\leq C\|f\|^{\beta}_{\mathcal{T}}\|f\|_{E}
\end{equation}
then 
\begin{equation}
\mathcal{E}_{\theta\beta, q}(E, \mathcal{T})\subset (F, F)^{K}_{\theta, q}, , \>0<\theta<1, \>0<q\leq \infty.
\end{equation}
\end{thm}\label{intthm}

Now we return to the situation on domains. 
Let $L$ be  a   self-adjoint positive definite operator   in  a Hilbert space $L_{2}(\Omega)$ which was introduced in the first section. We consider its positive root
$L^{1/2}$ and let 
 $\mathcal{D}^{r}, r\in \mathbb{R}_{+},$ be the domain of the operator $L^{r/2},  r\in \mathbb{R}_{+},$ with the graph norm $(I+L)^{r/2}$.

The inhomogeneous  Besov space $ \mathbf{B}_{2,q}^\alpha(\sqrt{L})$ is introduced as an interpolation space between the Hilbert space $L_{2}(\Omega)$ and Sobolev space $\mathcal D^{r}$ where
 $r$ can be any natural number such that $0<\alpha<r, 1\leq
q\leq\infty$. Namely, we have 
$$
\mathbf{B}^{\alpha}_{2,q}(\sqrt{L})=( L_{2}(\Omega),\mathcal{D}^{r})^{K}_{\theta, q},\>\>\> 0<\theta=\alpha/r<1,\>\>\>
1\leq q\leq \infty.
$$
where $K$ is the Peetre's interpolation functor.

We introduce a notion of best
approximation
\begin{equation}
\mathcal{E}(f,\omega)=\inf_{g\in
E_{\omega}(\sqrt{L})}\|f-g\|_{L_{2}(\Omega)}.\label{BA1}
\end{equation}

Our goal is to apply Theorem \ref{equivalence} in the situation where $E$ is the linear space $L_{2}(\Omega)$ with its regular norm,  $\>F$ is the linear space $\mathcal {D}^{r},$ with the graph norm $(I+L)^{r/2}$, and $\mathcal{T}=E_{\omega}(\sqrt{L})=E_{\omega^{2}}(L)$ is a natural abelian group of finite sequences of Fourier coefficients $\mathbf{c}=(c_{1},...c_{m})\in E_{\omega}(\sqrt{L})$ where $m$ is the greatest index such that the eigenvalue $\lambda_{m}\leq \omega$.
The quasi-norm $\|\mathbf{c}\|_{E_{\omega}(\sqrt{L})} $ where $\mathbf{c}=(c_{1},...c_{m})\in E_{\omega}(\sqrt{L})$  is defined as square root from the highest eigenvalue  $\lambda_{j}$ for which the corresponding Fourier coefficient $c_{j}\neq 0$ but $c_{j+1}=...=c_{m}=0$:
$$
\|\mathbf{c}\|_{E_{\omega}(L)} =\|(c_{1},...c_{m})\|_{E_{\omega}(L)}=\max\left\{\sqrt{\lambda_{j}}: c_{j}\neq 0, \>\>c_{j+1}=...=c_{m}=0\right\}.
$$

\begin{rem}
Let us stress that  the  reason we need language of  quasi-normed spaces is because $\|\mathbf{c}\|_{E_{\omega}(L)} $ is not a norm but only a quasi-norm on $ E_{\omega}(\sqrt{L})$.
\end{rem}
By using Plancherel Theorem it is easy to verify a  generalization of the Bernstein inequality for bandlimited functions $f\in E_{\omega}(\sqrt{L})$:
$$
\|L^{r}f\|_{L_{2}(\Omega)}\leq \omega ^{r}\|f\|_{L_{2}(\Omega)},\>\>\>r\in \mathbb{R}_{+},
$$
and  an analog  of the Jackson inequality:
$$
\mathcal{E}(f,\omega)\leq \omega^{-r}\|L^{r}f\|_{L_{2}(\Omega)},\>\>\>r\in \mathbb{R}_{+}.
$$

These two inequalities and Theorem \ref{equivalence}  imply the following result (compare to  \cite{Pes88}, \cite{Pes09}, \cite{Pes11}).

\begin{thm} \label{approx}
The norm of the Besov space $\mathbf{B}_{2,q}^{\alpha}(\sqrt{L}), \alpha>0, 1\leq q\leq\infty$ is
equivalent to the following norm
\begin{equation}
\|f\|_{L_{2}(\Omega)}+\left(\sum_{k=0}^{\infty}\left(2^{k\alpha }\mathcal{E}(f,
2^{k})\right)^{q}\right)^{1/q}.
\end{equation}
\label{maintheorem1}
\end{thm}
Let	function $F_{j}$ be the same as in subsection \ref{cutoff} and  
\begin{equation}\label{range}
F_j\left(\sqrt{L}\right):  L_{2}(\Omega)  \rightarrow E_{[2^{j-1},  2^{j+1}]}(\sqrt{L}),\>\>\> \left\|F_j\left(\sqrt{L}\right)\right\|\leq1.
 \end{equation}

    \begin{thm}\label{mainLemma1}
     The norm of the Besov space $\mathbf{B}_{2,q}^{\alpha}(\sqrt{L})$  for $\alpha>0, 1\leq
q\leq \infty$ is equivalent to
\begin{equation} 
\left(\sum_{j=0}^{\infty}\left(2^{j\alpha
}\left \|F_j\left(\sqrt{L}\right)f\right \|_{L_{2}(\Omega)}\right)^{q}\right)^{1/q},
\label{normequiv-1}
\end{equation}
  with the standard modifications for $q=\infty$.
\end{thm}

\begin{proof}

In the same notations as above the following version of Calder\'on decomposition holds:
$$
\sum_{j\in \mathbb{N}} F_j\left(\sqrt{L}\right)f= f,\> \>\>f\in L_{2}(\Omega).
$$
We obviously have
$$
\mathcal{E}(f, 2^{k})\leq \sum_{j> k} \left \|F_j\left(\sqrt{L}\right)f\right \|_{L_{2}(\Omega)}.
$$
By using the discrete Hardy inequality \cite{DeV} we obtain the estimate
\begin{equation} \label{direct}
\|f\|+\left(\sum_{k=0}^{\infty}\left(2^{k\alpha }\mathcal{E}(f,
2^{k})\right)^{q}\right)^{1/q}\leq C \left(\sum_{j=0}^{\infty}\left(2^{j\alpha
}\left \|F_j\left(\sqrt{L}\right)f\right \|_{L_{2}(\Omega)}\right)^{q}\right)^{1/q}
\end{equation}
Conversely, 
 for any $g\in E_{2^{j-1}}(\sqrt{L})$ we have
$$
\left\|F_j\left(\sqrt{L}\right)f\right\|_{L_{2}(\Omega)}=\left\|F_{j}\left(\sqrt{L}\right)(f-g)\right\|_{L_{2}(\Omega)}\leq \|f-g\|_{L_{2}(\Omega)}.
$$
It gives the inequality 
$$
\left\|F_j\left(\sqrt{L}\right)f\right\|_{L_{2}(\Omega)}\leq \mathcal{E}(f,\>2^{j-1}),
$$
which shows that the inequality opposite to (\ref{direct}) holds.
 This completes the proof.

\end{proof}

 \begin{thm}\label{mainLemma2}
     The norm of the Besov space $\mathbf{B}_{2,q}^{\alpha}(\sqrt{L})$  for $\alpha>0, 1\leq
q\leq \infty$ is equivalent  to
\begin{equation} 
 \left(\sum_{j=0}^{\infty}2^{j\alpha q }
\left(\sum_{i=1}^{I_{j}}\left|\left<f,\varphi_{j,i}\right>\right|^{2}\right)^{q/2}\right)^{1/q},
\label{normequiv}
\end{equation}
  with the standard modifications for $q=\infty$.
\end{thm}
\begin{proof} 

According to (\ref{???}) we have \begin{equation}
(1-\delta)\left \|F_j\left(\sqrt{L}\right)f\right \|_{L_{2}(\Omega)}^{2}\leq 
\sum _{ i=1}^{I_{j}}\left|\left< F_{j}\left({ \sqrt{L}}\right)f, \Phi_{j,i}\right>\right|^{2}\leq
\left \|F_j\left(\sqrt{L}\right)f\right \|_{L_{2}(\Omega)}^{2},
\end{equation}
where  $F_j\left(\sqrt{L}\right)f\in  E_{2^{j+1}}(\sqrt{L})$. Since $\varphi_{j,i}=F_{j}\left(\sqrt{L}\right)\Phi_{j,i}$ we obtain for any $f\in L_{2}(\Omega)$ 
$$
\sum_{i=1}^{I_{j}}\left|\left<f,\varphi_{j,i}\right>\right|^{2}\leq \left \|F_j\left(\sqrt{L}\right)f\right \|_{L_{2}(\Omega)}^{2}\leq \frac{1}{1-\delta}\sum_{i=1}^{I_{j}}\left|\left<f,\varphi_{j,i}\right>\right|^{2}.
$$
Theorem is proved. 
\end{proof}

\end{document}